\documentclass{article}

 \PassOptionsToPackage{numbers, compress}{natbib}

 \usepackage[preprint]{nips_2018}

\usepackage[utf8]{inputenc} 
\usepackage[T1]{fontenc}    
\usepackage{hyperref}       
\usepackage{url}            
\usepackage{booktabs}       
\usepackage{amsfonts}       
\usepackage{nicefrac}       
\usepackage{microtype}      

\usepackage[utf8]{inputenc} 
\usepackage[english]{babel} 

\usepackage{graphicx}
\usepackage{mathtools, cuted}
\usepackage{hyperref}
\usepackage{amsmath}
\usepackage{amssymb}
\usepackage{mathtools}
\usepackage{amsthm}

\newtheorem{theorem}{Theorem}[section]
\newtheorem{definition}[theorem]{Definition}
\newtheorem{corollary}[theorem]{Corollary}
\newtheorem{lemma}[theorem]{Lemma} 
\newtheorem{proposition}[theorem]{Proposition} 
\newtheorem{remark}[theorem]{Remark}

\title{Markov-Induced CM Model}

\author{
  Reza Rezaie and X. Rong Li\\
 Department of Electrical Engineering\\
 University of New Orleans\\
New Orleans, LA 70148 \\
  \texttt{rrezaie@uno.edu} and \texttt{xli@uno.edu} \\
}

\begin{document}

\maketitle

\begin{abstract}
Conditionally Markov (CM) sequences are powerful mathematical tools for modeling random phenomena. There are several classes of CM sequences one of which is the reciprocal sequence. Reciprocal sequences have been widely used in many areas including image processing, intelligent systems, and acausal systems. To use them in application, we need not only their applicable dynamic models, but also some general approaches to designing parameters of dynamic models. Dynamic models governing two important classes of nonsingular Gaussian (NG) CM sequences (called $CM_L$ and $CM_F$ models), and a dynamic model governing the NG reciprocal sequence (called reciprocal $CM_L$ model) were presented in our previous work. In this paper, these models are studied in more detail and general approaches are presented for their parameter design. It is shown that every reciprocal $CM_L$ model can be induced by a Markov model and parameters of the reciprocal $CM_L$ model can be obtained from those of the Markov model. Also, it is shown how NG CM sequences can be represented in terms of a NG Markov sequence and an independent NG vector. This representation provides a general approach for parameter design of $CM_L$ and $CM_F$ models. In addition, it leads to a better understanding of CM sequences, including the reciprocal sequence.

\end{abstract}


\section{Introduction}

Reciprocal processes have been applied to many different areas, including stochastic mechanics, image processing, acausal systems, trajectory modeling and intent inference, and intelligent systems (e.g., \cite{Levy_1}--\cite{DW_Conf}). In \cite{Levy_1}--\cite{Levy_2}, reciprocal processes were discussed in the context of stochastic mechanics. \cite{Picci}--\cite{Picci2} considered the application of reciprocal processes in image processing. In \cite{Krener1} the behavior of acausal systems was described using reciprocal processes. Based on quantized state space, \cite{Fanas1}--\cite{White_T} used finite-state reciprocal sequences for detection of anomalous trajectory pattern, intent inference, and tracking. \cite{Simon} used the idea of reciprocal processes for intent inference in intelligent interactive displays of vehicles. \cite{DD_Conf}--\cite{DW_Conf} proposed some classes of CM sequences (including reciprocal) for trajectory modeling.

Since they are desirable for many applications, reciprocal processes should be studied more to develop easily applicable tools for their application. This is the main goal of this paper. 

Consider stochastic sequences defined over the interval $[0,N]=\lbrace 0,1,\ldots,N \rbrace$. For convenience, let the index be time. A sequence is Markov iff conditioned on the state at any time $k$, the subsequences before and after $k$ are independent. A sequence is reciprocal iff conditioned on the states at any two times $k_1$ and $k_2$, the subsequences inside and outside the interval $[k_1,k_2]$ are independent. In other words, inside and outside are independent given the boundaries. A sequence is CM over $[0,N]$ iff conditioned on the state at time $0$ ($N$), the sequence is Markov over $(0,N]$ ($[0,N)$). The Markov sequence and the reciprocal sequence are two important classes of CM sequences.

CM processes were introduced in \cite{Mehr} for Gaussian processes based on mean and covariance functions. Also, stationary Gaussian CM processes were characterized, and construction of some non-stationary Gaussian CM processes was discussed. \cite{ABRAHAM} extended the definition of Gaussian CM processes (presented in \cite{Mehr}) to the general (Gaussian/non-Gaussian) case. The relationship between the CM process and the reciprocal process was studied in \cite{CM_Part_II_A_Conf}, \cite{ABRAHAM}, \cite{Carm}. Moreover, it was shown how continuous time Gaussian CM processes can be represented in terms of a Wiener process and an uncorrelated Gaussian random vector \cite{ABRAHAM}, \cite{Mehr}. In \cite{CM_Part_I}, different (Gaussian/non-Gaussian) CM sequences were defined based on conditioning at the first or the last index (time) of the CM interval, and (stationary/non-stationary) NG CM sequences were studied. Also, dynamic models and characterizations of NG CM sequences were presented. 

Reciprocal processes were introduced in \cite{Berstein} in connection to a problem posed by E. Schrodinger \cite{Schrodinger_1}--\cite{Schrodinger_2}. Later, they were studied more in \cite{Slepian}--\cite{CM_Explicitly} and others. A dynamic model and a characterization of the NG reciprocal sequence were presented in \cite{Levy_Dynamic}. It was shown that the evolution of a reciprocal sequence can be described by a second-order nearest-neighbor model driven by locally correlated dynamic noise \cite{Levy_Dynamic}. That model is interesting since it can be considered a generalization of the Markov model. However, due to the correlation of the dynamic noise as well as the nearest-neighbor structure, it is not necessarily easy to apply \cite{DD_Conf}--\cite{DW_Conf}. A dynamic model (governing the NG reciprocal sequence) was presented in \cite{CM_Part_II_A_Conf} from the viewpoint of the CM sequence. This dynamic model, called the reciprocal $CM_L$ dynamic model, is driven by white noise and is easily applicable. 

The contributions of this paper are as follows. In \cite{CM_Part_II_A_Conf}, it was not discussed how the parameters of the reciprocal $CM_L$ model can be designed in application. In this paper, an approach is presented for designing these parameters. Also, inspired by \cite{ABRAHAM}, a representation of NG CM sequences in terms of a NG Markov sequence and an uncorrelated NG vector is obtained. In addition, construction of a NG CM sequence based on a NG Markov sequence and an uncorrelated NG vector is presented. This representation and construction of CM sequences can be used for designing parameters of $CM_L$ and $CM_F$ dynamic models based on those of the Markov model. Besides providing a better insight into CM and reciprocal sequences, the results of this paper make these sequences more easily applicable.

The paper is organized as follows. Section \ref{Definitions_Preliminaries} reviews some definitions and results required for later sections. In Section \ref{Section_Models}, an approach is presented for design of parameters of the reciprocal $CM_L$ model. In Section \ref{Section_CM_Markov}, a representation (construction) of a NG CM sequence as a sum of a NG Markov sequence and an uncorrelated NG vector is discussed. Section \ref{Section_Summary_Conclusions} contains a summary and conclusions.

\section{Definitions and Preliminaries}\label{Definitions_Preliminaries}

\subsection{Conventions}\label{Subsection_Convention}

Throughout the paper we consider stochastic sequences defined over the interval $[0,N]$, which is a general discrete index interval. For convenience this discrete index is called time. The following conventions are used throughout the paper:
\begin{align*}
[i,j]& \triangleq \lbrace i,i+1,\cdots ,j-1,j \rbrace\\
(i,j)& \triangleq \lbrace i+1,i+2,\cdots ,j-2,j-1 \rbrace\\
[x_k]_{i}^{j} & \triangleq \lbrace x_k, k \in [i,j] \rbrace\\
[x_k] & \triangleq [x_k]_{0}^{N}\\
[x_k]_{J} & \triangleq \lbrace x_k, k \in J \rbrace, \quad J \subset [0,N]\\
i,j, k_1,k_2, l_1, l_2& \in [0,N], k_1<k_2, i<j 
\end{align*}
where $k$ in $[x_k]_i^j$ is a dummy variable. We use ``$ \setminus $" for set subtraction. $C_{l_1,l_2}$ is a covariance function. $C$ is the covariance matrix of the whole sequence $[x_k]$. Also, $0$ may denote a zero scalar, vector, or matrix, as is clear from the context. $F(\cdot | \cdot)$ denotes the conditional cumulative distribution function (CDF).

The abbreviations ZMNG and NG are used for ``zero-mean nonsingular Gaussian" and ``nonsingular Gaussian", respectively.

\subsection{Definitions and Notations}\label{Definitions}

CM sequences are defined as follows. A sequence $[x_k]$ is $CM_c, c \in \lbrace 0,N \rbrace$, iff conditioned on the state at time $0$ (or $N$), the sequence is Markov over $(0,N]$ ($[0,N)$). The above definition is equivalent to the following lemma \cite{CM_Part_I}.

\begin{lemma}\label{CMc_CDF}
$[x_k]$ is $CM_c, c \in \lbrace 0,N \rbrace$, iff 
 \begin{align}
 F(\xi _k|[x_{i}]_{0}^{j},x_{c})=F(\xi _k|x_j,x_c)\label{CDF_1}
 \end{align} 
for every $j,k \in [0,N], j<k$, and every $\xi _k \in \mathbb{R}^d$, where $d$ is the dimension of $x_k$.  

\end{lemma}

\begin{remark}\label{R_CMN}
For the forward direction, we have
\begin{align*}
CM_c=\left\{ \begin{array}{cc} 
CM_F & \text{if} c=0\\
CM_L & \text{if} c=N
\end{array} \right.
\end{align*}
where the subscript ``$F$" or ``$L$" is used because the conditioning is at the \textit{first} or \textit{last} time.

\end{remark}

Forward and backward directions are parallel. So we only consider the forward one.

The reciprocal sequence is defined as follows. A sequence is reciprocal iff the subsequences inside and outside the interval $[k_1,k_2]$ are independent conditioned on the boundaries $x_{k_1}$ and $x_{k_2}$ ($\forall k_1,k_2 \in [0,N]$). The above definition is equivalent to the following lemma \cite{CM_Part_II_A_Conf}, \cite{Jamison_1}.

\begin{lemma}\label{CDF}
$[x_k]$ is reciprocal iff 
\begin{align}
F(\xi _k|[x_{i}]_{0}^{j},[x_i]_l^N)=F(\xi _k|x_j,x_l)\label{CDF_1}
\end{align} 
for every $j,k,l \in [0,N]$ ($j < k < l$), and every $\xi _k \in \mathbb{R}^d$, where $d$ is the dimension of $x_k$.  

\end{lemma}

\begin{lemma}\label{CMc_CDF}
$[x_k]$ is Markov iff 
 \begin{align}
 F(\xi _k|[x_{i}]_{0}^{j})=F(\xi _k|x_j)\label{CDF_1}
 \end{align} 
for every $j,k \in [0,N], j<k$, and every $\xi _k \in \mathbb{R}^d$, where $d$ is the dimension of $x_k$.  

\end{lemma}

\subsection{Preliminaries}\label{Preliminaries}

We review some results for later sections \cite{CM_Part_I}, \cite{CM_Part_II_A_Conf}, \cite{Levy_Dynamic}, \cite{Ackner}. 

\begin{definition}\label{CMc_Matrix}
A symmetric positive definite matrix is called $CM_L$ if it has form $\eqref{CML}$ and $CM_F$ if it has form $\eqref{CMF}$.
\begin{align}\left[
\begin{array}{ccccccc}
A_0 & B_0 & 0 & \cdots & 0 & 0 & D_0\\
B_0' & A_1 & B_1 & 0 & \cdots & 0 & D_1\\
0 & B_1' & A_2 & B_2 & \cdots & 0 & D_2\\
\vdots & \vdots & \vdots & \vdots & \vdots & \vdots & \vdots\\
0 & \cdots & 0 & B_{N-3}' & A_{N-2}  & B_{N-2} & D_{N-2}\\
0 & \cdots & 0 & 0 & B_{N-2}' & A_{N-1} & B_{N-1}\\
D_0' & D_1' & D_2' & \cdots & D_{N-2}' & B_{N-1}' & A_N
\end{array}\right]\label{CML}
\end{align}
\begin{align}\left[
\begin{array}{ccccccc}
A_0 & B_0 & D_2 & \cdots & D_{N-2} & D_{N-1} & D_{N}\\
B_0' & A_1 & B_1 & 0 & \cdots & 0 & 0\\
D_2' & B_1' & A_2 & B_2 & \cdots & 0 & 0\\
\vdots & \vdots & \vdots & \vdots & \vdots & \vdots & \vdots\\
D_{N-2}' & \cdots & 0 & B_{N-3}' & A_{N-2}  & B_{N-2} & 0\\
D_{N-1}' & \cdots & 0 & 0 & B_{N-2}' & A_{N-1} & B_{N-1}\\
D_{N}' & 0 & 0 & \cdots & 0 & B_{N-1}' & A_N
\end{array}\right]\label{CMF}
\end{align}

\end{definition}

\begin{definition}\label{CMc_Matrix}
A matrix is called cyclic tri-diagonal if it has form $\eqref{Cyclic_Tridiagonal}$.
\begin{align}\left[
\begin{array}{ccccccc}
A_0 & B_0 & 0 & \cdots & 0 & 0 & D_0\\
B_0' & A_1 & B_1 & 0 & \cdots & 0 & 0\\
0 & B_1' & A_2 & B_2 & \cdots & 0 & 0\\
\vdots & \vdots & \vdots & \vdots & \vdots & \vdots & \vdots\\
0 & \cdots & 0 & B_{N-3}' & A_{N-2}  & B_{N-2} & 0\\
0 & \cdots & 0 & 0 & B_{N-2}' & A_{N-1} & B_{N-1}\\
D_0' & 0 & 0 & \cdots & 0 & B_{N-1}' & A_N
\end{array}\right]\label{Cyclic_Tridiagonal}
\end{align}

\end{definition}

\begin{definition}\label{CMc_Matrix}
A matrix is called tri-diagonal if it has form $\eqref{Tridiagonal}$.
\begin{align}\left[
\begin{array}{ccccccc}
A_0 & B_0 & 0 & \cdots & 0 & 0 & 0\\
B_0' & A_1 & B_1 & 0 & \cdots & 0 & 0\\
0 & B_1' & A_2 & B_2 & \cdots & 0 & 0\\
\vdots & \vdots & \vdots & \vdots & \vdots & \vdots & \vdots\\
0 & \cdots & 0 & B_{N-3}' & A_{N-2}  & B_{N-2} & 0\\
0 & \cdots & 0 & 0 & B_{N-2}' & A_{N-1} & B_{N-1}\\
0 & 0 & 0 & \cdots & 0 & B_{N-1}' & A_N
\end{array}\right]\label{Tridiagonal}
\end{align}

\end{definition}

$A_k$, $B_k$, and $D_k$ are matrices in general.  

\begin{remark}
We use $CM_c$ to mean either $CM_L$ or $CM_F$ or both.

\end{remark}

\begin{theorem}\label{CML_Characterization} 
(i) A NG sequence with covariance matrix $C$ is $CM_c$ iff $C^{-1}$ has the $CM_c$ form.

(ii) A NG sequence with covariance matrix $C$ is reciprocal iff $C^{-1}$ is cyclic tri-diagonal.

(iii) A NG sequence with covariance matrix $C$ is Markov iff $C^{-1}$ is tri-diagonal.

\end{theorem}

The following corollary follows from the characterizations presented in Theorem \ref{CML_Characterization} \cite{CM_Part_II_A_Conf}.

\begin{corollary}\label{CMLCMF_Reciprocal}
A NG sequence is reciprocal iff it is both $CM_L$ and $CM_F$.
\end{corollary}

A $CM_c$ dynamic model is as follows.

\begin{theorem}\label{CML_Dynamic_Forward_Proposition}
A ZMNG $[x_k]$ with covariance function $C_{l_1,l_2}$ is $CM_c$ iff it is governed by
\begin{align}
x_k=G_{k,k-1}x_{k-1}+G_{k,c}x_c+e_k, \quad k \in (0,N] \setminus \lbrace c \rbrace
\label{CML_Dynamic_Forward}
\end{align}
where $[e_k]$ is a zero-mean white Gaussian sequence with nonsingular covariances $G_k$, along with boundary condition\footnote{$e_0$ and $e_N$ in $\eqref{CML_Forward_BC1}$ are not necessarily the same as $e_0$ and $e_N$ in $\eqref{CML_Forward_BC2}$. Just for simplicity we use the same notation.}
\begin{align}
&x_0=e_0, \quad x_c=G_{c,0}x_0+e_c \, \, (\text{for} \,\, c=N)\label{CML_Forward_BC1}
\end{align}
or equivalently
\begin{align}
&x_c=e_c, \quad x_0=G_{0,c}x_c+e_0 \, \, (\text{for} \, \, c=N) \label{CML_Forward_BC2}
\end{align}

\end{theorem}

The following theorem presents the reciprocal $CM_c$ models \cite{CM_Part_II_A_Conf}.
\begin{theorem}\label{CML_R_Dynamic_Forward_Proposition}
A ZMNG $[x_k]$ with covariance function $C_{l_1,l_2}$ is reciprocal iff it is governed by $\eqref{CML_Dynamic_Forward}$ along with $\eqref{CML_Forward_BC1}$ or $\eqref{CML_Forward_BC2}$, and 
\begin{align}
G_k^{-1}G_{k,c}=G_{k+1,k}'G_{k+1}^{-1}G_{k+1,c}
\label{CML_Condition_Reciprocal}
\end{align}
$\forall k \in (0,N-1)$ for $c=N$, and $\forall k \in (1,N)$ for $c=0$. Moreover, $[x_k]$ is Markov iff additionally we have, for $c=N$,
\begin{align}
G_N^{-1}G_{N,0}=G_{1,N}'G_{1}^{-1}G_{1,0}\label{CML_M_BC1}
\end{align}
for $\eqref{CML_Forward_BC1}$ or
\begin{align}
G_0^{-1}G_{0,N}=G_{1,0}'G_1^{-1}G_{1,N}\label{CML_M_BC2}
\end{align}
for $\eqref{CML_Forward_BC2}$; for $c=0$, we have 
\begin{align}
G_{N,0}=0\label{CMF_M}
\end{align}

\end{theorem}

\section{A Dynamic Model of Reciprocal Sequences}\label{Section_Models}

Based on Theorem \ref{CML_R_Dynamic_Forward_Proposition}, we can determine whether a given $CM_c$ model governs a reciprocal sequence or not. Also, Theorem \ref{CML_R_Dynamic_Forward_Proposition} gives the required conditions on the parameters of a reciprocal $CM_c$ model for their design. Theorem \ref{CML_R_Dynamic_FQ_Proposition} below provides an approach to parameter design of the reciprocal $CM_L$ model. We have a similar result for a backward $CM_L$ model (which is useful for the $CM_F$ sequence, because a $CM_F$ sequence is backward $CM_L$).

\begin{theorem}\label{CML_R_Dynamic_FQ_Proposition} 
 A ZMNG sequence $[x_k]$ with covariance function $C_{l_1,l_2}$ is reciprocal iff it is governed by $\eqref{CML_Dynamic_Forward}$ (for $c=N$) along with $\eqref{CML_Forward_BC1}$ or $\eqref{CML_Forward_BC2}$, where $(G_{k,k-1},G_{k,N},G_k)$, $k \in (0,N)$, are given by
\begin{align}
G_{k,k-1}&=M_{k,k-1}-G_{k,N}M_{N|k-1} \label{CML_Choice_1}\\
G_{k,N}&=G_kM_{N|k}'C_{N|k}^{-1} \label{CML_Choice_2}\\
G_k&=(M_{k}^{-1}+M_{N|k}'C_{N|k}^{-1}M_{N|k})^{-1}\label{CML_Choice_3}\\
M_{N|k}&=M_{N,N-1}\cdots M_{k+1,k}, \quad k \in (0,N)\nonumber
\end{align}
and $M_{N|N}=I$. Also,
\begin{align*}
C_{N|k}=\sum _{n=k}^{N-1}M_{N|n+1}M_{n+1}M_{N|n+1}', \quad k \in (0,N)
\end{align*}
where $M_{k,k-1}$, $ k \in (0,N]$, are square matrices, and $M_k$, $k \in (0,N]$, are positive definite of the dimension of $x_k$.

\end{theorem}
\begin{proof}
We only show how $\eqref{CML_Choice_1}$--$\eqref{CML_Choice_3}$ are obtained. More detail can be found in \cite{CM_Part_II_B}.

Given the square matrices $M_{k,k-1}$ and the positive definite matrices $M_{k}, k \in (0,N]$, there exists a ZMNG Markov sequence $[y_k]$ governed by 
\begin{align}\label{1Order_Markov}
y_k=M_{k,k-1}y_{k-1}+e^M_{k}, \quad k \in (0,N]
\end{align}
where $[e^M_k]_1^{N}$ is a zero-mean white Gaussian sequence with nonsingular covariances $M_k$, uncorrelated with $y_0$ with nonsingular covariance $C^y_0$. 

Since every Markov sequence is $CM_L$, one can obtain the following model for $[y_k]$
\begin{align}
y_k=G_{k,k-1}y_{k-1}+G_{k,N}y_N+e^y_k, \quad k \in (0,N)\label{CML_Dynamic_for_Markov}
\end{align}
where $[e^y_k]$ is a zero-mean white Gaussian sequence with nonsingular covariances $G_k$, and boundary condition
\begin{align}
y_N&=e^{y}_N, \quad y_0=G^{y}_{0,N}y_N+e^{y}_0\label{CML_R_FQ_BC2}
\end{align}

Parameters of model $\eqref{CML_Dynamic_for_Markov}$ can be obtained as follows. Since
\begin{align*}
p(y_k|y_{k-1})=\mathcal{N}(y_k;M_{k,k-1}y_{k-1},M_{k})
\end{align*}
by the Markov property, we have
\begin{align}
p(y_k|y_{k-1},y_N)&=\frac{p(y_k|y_{k-1})p(y_N|y_k)}{p(y_N|y_{k-1})}\label{CML_Reciprocal_Transition}\\
&=\mathcal{N}(y_k;G_{k,k-1}y_{k-1}+G_{k,N}y_N;G_k)\nonumber
\end{align}
$\forall k \in (0,N)$, and $G_{k,k-1}$, $G_{k,N}$, and $G_k$ are calculated by $\eqref{CML_Choice_1}$--$\eqref{CML_Choice_3}$, where for $i \in [0,N)$ we have
\begin{align*}
p(y_N|y_i)=\mathcal{N}(y_N;M_{N|i}y_{i},C_{N|i})
\end{align*}

Now, we construct a sequence $[x_k]$ governed by the same model $\eqref{CML_Dynamic_for_Markov}$ as
\begin{align}
x_k=G_{k,k-1}x_{k-1}+G_{k,N}x_N+e_k, \quad k \in (0,N)\label{CML_Dynamic_for_Markov_x}
\end{align}
and boundary condition
\begin{align}
x_N&=e_N, \quad x_0=G_{0,N}x_N+e_0\label{CML_R_FQ_BC2_x}
\end{align}
Note that parameters of $\eqref{CML_Dynamic_for_Markov}$ and $\eqref{CML_Dynamic_for_Markov_x}$ are the same (i.e., $G_{k,k-1}, G_{k,N}, G_k, k \in (0,N)$), but parameters of $\eqref{CML_R_FQ_BC2}$ (i.e., $G^y_{0,N}, G^y_0,G^y_N$) and $\eqref{CML_R_FQ_BC2_x}$ (i.e., $G_{0,N}, G_0,G_N$) are different. By Theorem \ref{CML_Dynamic_Forward_Proposition}, $[x_k]$ is a ZMNG $CM_L$ sequence. Then, based on the above results, necessity and sufficiency of the theorem are proved \cite{CM_Part_II_B}.
\end{proof}

By Theorem \ref{CML_R_Dynamic_FQ_Proposition}, we can first design $M_{k,k-1}$ and $M_k$, $k \in (0,N]$, and then parameters of the reciprocal $CM_L$ model. This approach is useful because design of the parameters of a Markov model is much easier (than of a reciprocal $CM_L$ model) and $M_{k,k-1}$ and $M_k$ can be seen as parameters of a Markov model. It is explained for the problem of motion trajectory modeling with destination information as follows. Let $[y_k]$ be a Markov sequence governed by $\eqref{1Order_Markov}$ (for example, a nearly constant velocity motion model without considering destination information) with parameters $(M_{k,k-1},M_{k}), k \in (0,N]$ (and the initial covariance $C^y_0$), and $[x_k]$ be a $CM_L$ sequence governed by $\eqref{CML_Dynamic_Forward}$ with parameters $(G_{k,k-1},G_{k,N},G_k)$, $k \in (0,N)$ given by $\eqref{CML_Choice_1}$--$\eqref{CML_Choice_3}$ (and any boundary condition $(G_0,G_{N,0},G_N)$ or $(G_N,G_{0,N},G_0)$). It can be seen that $[x_k]$ and $[y_k]$ have the same $CM_L$ evolution law ($\eqref{CML_Dynamic_for_Markov}$ and $\eqref{CML_Dynamic_for_Markov_x}$) while $[x_k]$ can have any joint endpoint distribution. Therefore, $[x_k]$ can model trajectories with any origin and destination distributions, where parameters of its $CM_L$ model are designed based on those of the Markov model. In the above example, on the one hand it is assumed that the moving object follows the Markov model, and on the other hand it is assumed that the destination (distribution) is known (this is the case in a real problem of trajectory modeling with destination information). Considering these two assumptions (i.e., motion Markov model and destination), the corresponding $CM_L$ model (with $\eqref{CML_Choice_1}$--$\eqref{CML_Choice_3}$) describes the behavior of the moving object satisfying both the above assumptions (also, see $\eqref{CML_Reciprocal_Transition}$). 

The dynamic model $\eqref{CML_Dynamic_for_Markov_x}$ is called a $CM_L$ model \textit{induced} by a Markov model. By Theorem \ref{CML_R_Dynamic_FQ_Proposition}, every $CM_L$ model induced by a Markov model is a reciprocal $CM_L$ model. Also, every reciprocal $CM_L$ model can be induced by a Markov model.

\section{Representations of CM and Reciprocal Sequences}\label{Section_CM_Markov}

A representation of continuous time Gaussian CM processes by a Wiener process and an uncorrelated Gaussian vector was presented in \cite{ABRAHAM}. Inspired by \cite{ABRAHAM}, we show how a NG CM sequence can be represented by a NG Markov sequence and an uncorrelated NG vector. Also, it is shown how to use a NG Markov sequence and an uncorrelated NG vector to construct a NG CM sequence. This representation not only gives more insight into the CM sequence, but also provides an approach for designing CM sequences in application. 

\begin{proposition}\label{CML_Markov_z_Proposition}
A ZMNG $[x_k]$ is $CM_c$ iff it can be represented as 
\begin{align}\label{CML_Markov_z}
x_k=y_k+ \Gamma _k x_c, \quad k \in [0,N] \setminus \lbrace c \rbrace
\end{align}
where $[y_k]_{[0,N] \setminus \lbrace c \rbrace}$ is a ZMNG Markov sequence, $x_c$ is a ZMNG vector uncorrelated with $[y_k]_{[0,N] \setminus \lbrace c \rbrace}$, and $\Gamma_k$ are some matrices. 

\end{proposition}
\begin{proof}
We only show that given a ZMNG Markov sequence $[y_k]_0^{N-1}$ uncorrelated with a ZMNG vector $x_N$, $[x_k]$ constructed by $x_k=y_k+\Gamma _kx_N, k \in [0,N)$ is a ZMNG $CM_L$ sequence, where $\Gamma_k$ are some matrices. More detail can be found in \cite{CM_Part_II_B}. It suffices to show that $[x_k]$ is governed by $\eqref{CML_Dynamic_Forward}$ and $\eqref{CML_Forward_BC2}$ ($c=N$).

Since $[y_k]_0^{N-1}$ is a ZMNG Markov sequence, we have
\begin{align}
y_k=M_{k,k-1}y_{k-1}+e^M_{k}, \quad k \in (0,N)\label{Markov_Model}
\end{align}
where $[e^M_{k}]_1^{N-1}$ is a zero-mean white Gaussian sequence with nonsingular covariances $M_{k}$, uncorrelated with $y_0$ with nonsingular covariance $C^y_0$. We have
\begin{align*}
x_0=y_0+\Gamma _0x_N
\end{align*}  
So, we consider $G_{0,N}=\Gamma _0$ and $e_0=y_0$. Then, for $k \in (0,N)$, we have
\begin{align*}
x_k&=y_k+\Gamma _kx_N\\
&=M_{k,k-1}y_{k-1}+e^M_{k}+\Gamma _kx_N\\
&=M_{k,k-1}x_{k-1}+(\Gamma _k-M_{k,k-1}\Gamma _{k-1})x_N+e^M_{k}
\end{align*}
Thus, for $k \in (0,N)$, we consider
\begin{align*}
&G_{k,k-1}=M_{k,k-1}\\
&G_{k,N}=\Gamma _k-M_{k,k-1}\Gamma _{k-1}\\
&e_k=e^M_{k}
\end{align*}
Nonsingularity of $[x_k]$ can be proved based on the nonsingularity of $[y_k]_0^{N-1}$ and $x_N$.
\end{proof}

Proposition \ref{CML_Markov_z_Proposition} makes a key concept behind the $CM_c$ sequence clear. Also, it provides an approach for design of $CM_c$ models in application. Below we explain the idea for designing a $CM_L$ model for motion trajectory modeling with destination information. The $CM_L$ model is more general than the reciprocal $CM_L$ model. Consequently, the following approach for parameter design is more general and flexible than that of Theorem \ref{CML_R_Dynamic_FQ_Proposition}. For example, it can incorporate other available information in the problem. The approach is as follows. First, a Markov model (e.g., a nearly constant velocity model) with the given origin distribution (without considering other information) is considered. The sequence governed by this model is $[y_k]_0^{N-1}$ in $\eqref{CML_Markov_z}$ (let $c=N$)). Assuming the destination, distribution of $x_N$, is known. Then, the matrices $\Gamma _k$ are designed according to the information available in the problem. In other words, based on $\Gamma _k$, the Markov sequence $[y_k]_0^{N-1}$ is modified to satisfy the available information leading to the desired trajectories which end up at the destination. A direct attempt for designing parameters of a $CM_L$ model for this problem is hard. However, the above approach based on Proposition \ref{CML_Markov_z_Proposition} provides some guidelines which make parameter design easier and intuitive.

A representation of the reciprocal sequence is as follows.

\begin{proposition}\label{Reciprocal_Markov_z}
A ZMNG $[x_k]$ is reciprocal iff it can be represented as both
\begin{align}
x_k&=y^L_k+\Gamma ^L_kx_N, \quad k \in [0,N) \label{CML_R}\\
x_k&=y^F_k+\Gamma ^F_kx_0, \quad  k \in (0,N]\label{CMF_R}
\end{align}
where $[y^L_k]_0^{N-1}$ and $[y^F_k]_1^N$ are ZMNG Markov sequences, $x_N$ and $x_0$ are ZMNG vectors uncorrelated with $[y^L_k]_0^{N-1}$ and $[y^F_k]_1^N$, respectively, and $\Gamma ^L_k$ and $\Gamma ^F_k$ are some matrices. 

\end{proposition}

\section{Summary and Conclusions}\label{Section_Summary_Conclusions}

An approach has been presented for designing parameters of reciprocal $CM_L$ dynamic models. This approach makes the nonsingular Gaussian (NG) reciprocal sequence more easily applicable. The approach is based on inducing a $CM_L$ model from a Markov model. It was shown that every $CM_L$ model induced by a Markov model is a reciprocal $CM_L$ model. Also, every reciprocal $CM_L$ model can be induced by a Markov model.

It has been shown how a NG CM (or reciprocal) sequence can be represented as a sum of a NG Markov sequence and an uncorrelated NG vector. In addition, construction of a NG CM sequence based on a NG Markov sequence and an uncorrelated NG vector is discussed. This representation and construction can be used for designing parameters of the CM dynamic models in application (based on designing parameters of a NG Markov model and an uncorrelated NG vector). The results of this paper lead to more insight into CM and reciprocal sequences, provide some tools for application of CM (and reciprocal) sequences, and demonstrate the benefits of studying the reciprocal sequence from the viewpoint of the CM sequence.

\subsubsection*{Acknowledgments}

Research was supported by NASA Phase03-06 through grant NNX13AD29A.

\end{document}